\documentclass[11pt]{amsart}

\usepackage{mystyle}

\title[Free Boundary On a Cone]{Free Boundary on a Cone}

\begin{document}
\author{Mark Allen \and
        H\'ector Chang Lara}

\maketitle

\begin{abstract}
We study two phase problems posed over a two dimensional cone generated by a smooth curve $\gamma$ on the unit sphere. We show that when $length(\gamma)<2\pi$ the free boundary avoids the vertex of the cone. When $length(\gamma)\geq2\pi$ we provide examples of minimizers such that the vertex belongs to the free boundary.
\end{abstract}


\section{Introduction}

The purpose of this paper is to initiate a study of free boundaries on manifolds with singularities. We study a free boundary problem that played a significant role in the historical development of the field of free boundary problems. The by now classical problem involves studying minimizers of the functional  
\begin{equation}   \label{e:j}
J(u,\W) = \int_\W |Du|^2 + \chi_{\{u>0\}}
\end{equation}
with a predetermined non negative boundary data. This situation appears in cavitational problems, flame propagation, optimal insulation among other models referenced for instance in the book \cite{Caffarellibook}. The Euler-Lagrange equation gives the following over determined problem for $u$,
\begin{alignat*}{2}
 \D u &= 0 &&\text{ in $\{u > 0\}\cap \W$},\\
 |Du^+| &= 1 &&\text{ in $\p\{u>0\}\cap \W$}.
\end{alignat*}
The regularity of $u$ and its free boundary $\p\{u>0\}$ was obtained by L. Caffarelli and H. Alt in \cite{Caffarelli81}. More complicated situations appear in the two phase problem where $E$ also penalizes the set where $u$ is negative, in which case the boundary data can have arbitrary sign and regularity estimates become more delicate, see \cite{Caffarelli87, Caffarelli89, Caffarelli88}. Arbitrary metrics with some regularity condition are considered by the series of papers by Sandro Salsa and Fausto Ferrari \cite{Salsa04,Salsa07,Salsa07-2,Salsa10}. See also \cite{DeSilva11} for an alternative and elegant approach.

In this paper we look at two phase problems with degenerate metrics. In terms of existence, the minimization problem can be solved in the functional space $H^1$ over manifolds with minimal assumptions of smoothness, for instance with corners. Our first attempt is to study the simplest case we could imagine, a two dimensional cone generated by a smooth simple closed curve $\gamma$ on the unit sphere. The main question of interest is to study the interaction of the free boundary with the vertex. 

The free boundary in the one phase problem given in \eqref{e:j} behaves similarly to minimal surfaces. For instance, it is well known that there are no nontrivial area minimizing cones in dimensions $n \leq 7$ while the Simons cone in dimension $n=8$ is area minimizing. Similarly, there are no minimizing cone solutions to \eqref{e:j} in dimensions $n=2,3$ (see \cite{MR2082392}) while in dimension $n=7$ a minimizing cone does exist (see \cite{MR2572253}) which is analogous to the Simons cone. In this paper we provide another connection between minimal surfaces and the free boundary arising from \eqref{e:j}. For distance minimizing geodesics on two dimensional cones (generated by a smooth simply connected curve $\gamma$ on the sphere) the following proposition is well-known
\begin{proposition}  \label{p:geo}
If $l = length(\gamma) < 2\pi$, no distance minimizing geodesics pass through the vertex. If $l = length(\gamma) \geq 2\pi$, then there are distance minimizing geodesics that pass through the vertex.
\end{proposition} 
The proof when $l < 2\pi$ can be found  
in Section 4-7 in the book \cite{DoCarmo76}. 

In this paper we prove the analogous result of Proposition \ref{p:geo} for minimizers of \eqref{e:j} on a cone. 
\begin{theorem}  \label{main}
Let $u$ be a minimizer of \eqref{e:j}. If $l < 2\pi$, then the vertex $0 \notin \partial \{u=0\}$. If $l \geq 2\pi$ the free boundary can pass through the vertex. 
\end{theorem}
The proof that the free boundary avoids the vertex when $l < 2\pi$ is given in Section 3. In Section 4 we provide examples of the free boundary passing through the vertex when $l \geq 2\pi$. Even more, for $l\geq4\pi$ there are examples of one phase problems where two different positive phases meet at the vertex which is an unexpected singular behavior.
 
After understanding the previous particular case we plan to continue studying the regularity of the free boundary with degenerate metrics in future works. At the moment we do not know about the optimal regularity of the solution when $l > 2\pi$ and the vertex belongs to the free boundary of both the positivity and negativity phases. Notice that from the Fourier series representation, harmonic functions over a cone with $l > 2\pi$ might be only H\"older continuos at the vertex. However we expect that minimizers which evaluate zero at the vertex to also be Lipschitz.

Other interesting directions to explore are:
\begin{enumerate}
\item Homogenization problems with singular metrics. Consider for instance the one phase problem posed over a manifold with many small corners. This might be related with the homogenization of capillary drops over inhomogeneous surfaces studied in \cite{Caffarelli07}.
\item Free boundary problems over higher dimensional cones. In this case our approach seems limited by the fact that we strongly used that outside of the vertex the metric can be considered flat.
\end{enumerate}

It is worth noting that Theorem \ref{main} also bears resemblance to the result obtained by H. Shahgholian in \cite{MR2120191} where the free boundary in the obstacle problem can enter into the corner of a fixed boundary if and only if the aperture of the corner is greater than or equal to $\pi$. Many of the techniques and methods developed in studying the classical problem \eqref{e:j} aided in the study of the obstacle problem. The results and techniques of this paper may aid in the future study of obstacle problems over rough obstacles which has applications in mathematical finance \cite{MR2753636}. 

The paper is organized as follow. In Section 2 we discuss existence, regularity and stability of the minimizers. The proofs of many of these statements are simple adaptations from the arguments found in the classical literature and are left for the appendix at the end. Section 3 is our main contribution. There we prove that, in the case $l < 2\pi$, the free boundary of our minimizers always avoid the vertex. Our approach consists in reducing the problem to find a better competitor against 1-homogeneous minimizers. Finally in Section 4 we discuss the situation when $l \geq 2\pi$. We provide some examples where the vertex belongs to the free boundary and for even larger values of $l$ we also show that more than one positive phase can meet at the vertex.


\section{Preliminaries}\label{sec:Preliminaries}

We fix, without loss of generality, our two dimensional cone $\cC \ss \R^3$ to have its vertex at the origin. Such a cone $\cC$ embedded in $\R^3$ is a ruled surface that inherits a flat metric. By this we mean that for every open set $U \ss \cC\sm\{0\}$ there is always a local isometry that maps it to an open set of $\R^2 \sm \{0\}$ with the flat metric. This follows from a parametrization of $\cC$ given by polar coordinates, since $\cC \cap B_1$ is just a one dimensional smooth simple closed curve that can be parametrized by arc length. In order to also have an injective isometry we can lift the previous map to the universal covering of $\R^2 \sm \{0\}$ which we denote by $\cR^2$. In polar coordinates $\cR^2$ gets parametrized by a radius and an angle $(r,\theta)\in\R^+\times\R$.

Let $l$ be the length of the trace of the given cone in the unit sphere. From now on we just say that $\cC$ has length $l$. This length gives us a canonical representation of $\cC\sm\{0\}$ as $\cR^2/\{\theta \in l\Z\}$ with the flat metric in $\cR^2$. We denote by $\phi_l$ the (isometry) quotient map going from $\cR^2$ to $\cR^2/\{\theta \in l\Z\}$.

A way to visualize what we have described so far is by cutting the cone by one of its rays starting at the origin and laying the surface flat, keeping in mind the identification at the boundary. In the case that $l<2\pi$ it looks like $\R^2$ minus a cone and in the case that $l > 2\pi$ we will have some overlap. See Figure \ref{f:cone}. 

\begin{figure}[h]  
  \includegraphics[width = 8cm]{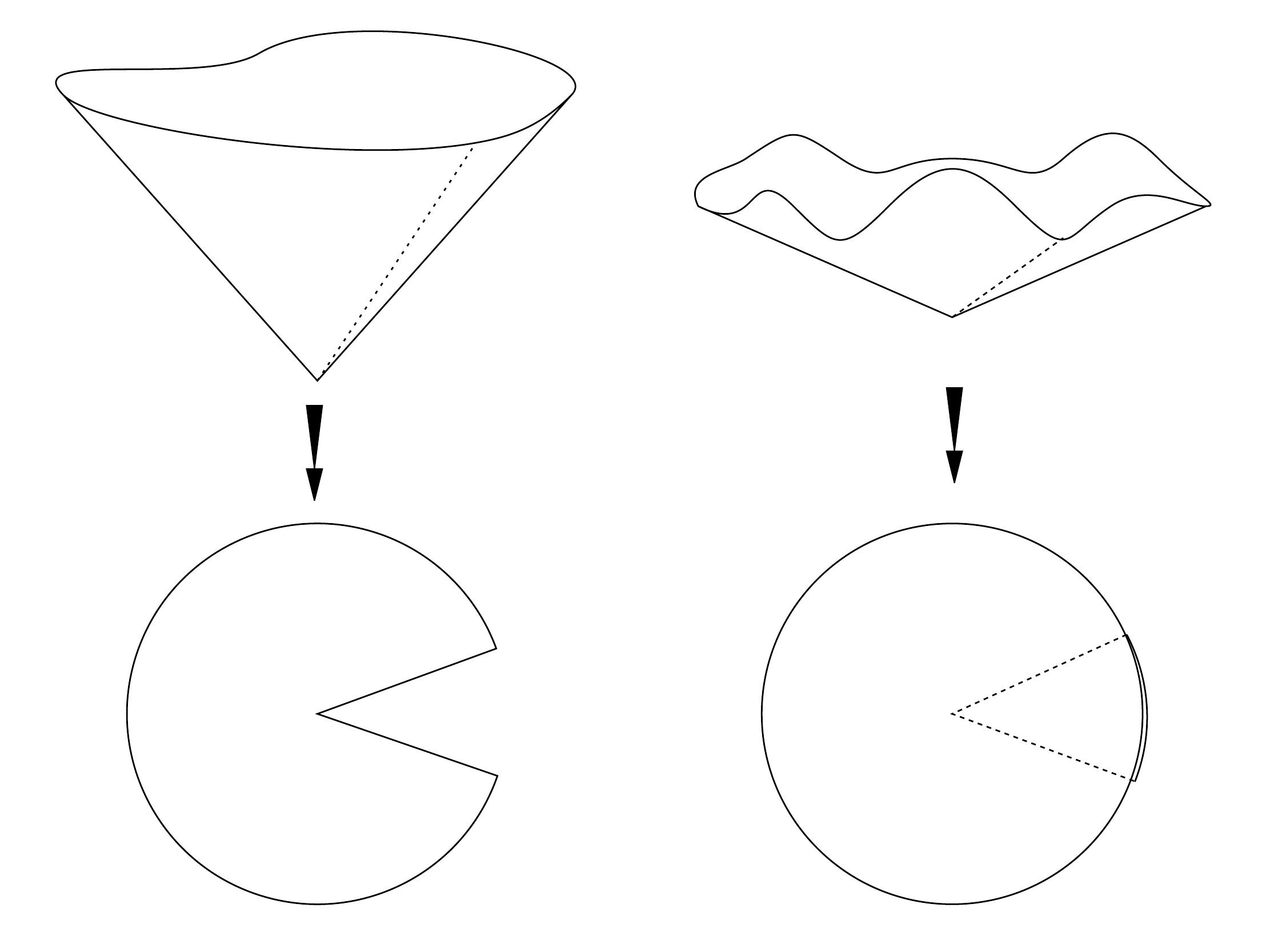}
  \caption{Cutting two cones and laying them flat}
  \label{f:cone}
\end{figure}

Notice that all we have said so far also holds for any two dimensional cone embedded in $\R^n$ with $n\geq 2$. After looking at the universal covering of such cone minus its vertex the domain gets fixed to a quotient of the form $\cR^2/\{\theta \in l\Z\}$.


\subsection{Harmonic functions on a Cone}

In this section we study some basic properties of harmonic functions on $\cC$. Given the previous discussion, we have that for $f:\W\ss\cC\to\R$ we can define any differential operator acting on $f$ in the distributional sense. A test function $\varphi(x)$ in this case is a smooth function in $\cC\sm\{0\}$ with all its derivatives uniformly bounded in $x$ and such that $\varphi(x)$ is continuous at the vertex. Notice that, because there is not a tangent plane at the vertex, we can not make sense of the gradient of a function at the vertex. However we can always ask if the function has a modulus of continuity even at the vertex. 

The following proposition gives some equivalent definitions for subharmonic functions. We omit the proof.

\begin{proposition}[Subharmonic functions]
For a function $h:\W \ss \cC \to \R$ the following are equivalent and in such cases we say that $h$ is subharmonic:
\begin{enumerate}
\item For every $K \ss \W$ compact, $h \in H^1(K)$ and it minimizes the Dirichlet energy $\int_{K} |Dh|^2$ over all the functions less or equal than $h$ in $K$ and with the same boundary data as $h$ in $\p K$. Here we denoted by $Df$ the tangential gradient and the integral is taken with respect to the area form in $\cC$.
\item $h \in L^1(\W)$ and it has the mean value property for subharmonic functions in $\W$.
\item Seeing as a function $h:\W' \ss \cR^2/\{\theta \in l\Z\} \to \R$ (where $\W\sm\{0\}$ gets mapped to $\W'$ by the isometry) $\D h \geq 0$ in $\W'$ in the sense of distributions and has the mean value property for subharmonic functions at the origin if $0 \in \W$.
\end{enumerate}
\end{proposition}

The definition of superharmonic functions is analogous to the previous one by changing the corresponding inequalities. A harmonic function is one which is both sub and super harmonic simultaneously. This definition in particular allows one to perform integration by parts and recover Green's formula even in the case where the vertex belongs to the domain of integration. 

The following proposition follows from the Fourier series representation and gives us already some intuition about how differently harmonic functions behave according to the length of the cone. We also omit its proof.

\begin{proposition}\label{prop:harmonic_functions}
Let $\C$ be a cone with length $l$. Any harmonic function $h$ on $\C$ may be written as 
\begin{align}  \label{e:expan}
h(r,\theta) = \sum_{k=0}^{\infty}{r^{2 \pi k/l} \left(a_k \cos \frac{2 \pi k}{l} \theta 
                                                \ + \ b_k \sin \frac{2 \pi k}{l} \theta \right) }.
\end{align}
In particular, $h \in C^{2\pi/l}$.
\end{proposition}


\subsection{Minimization problem}

Here we give the explicit minimization problem we want to study and show existence in $H^1$.

Given a domain $\W \ss \cC$ and $\l_+\neq\l_-$ non negative numbers, let $J:H^1(\W)\to\R$ given by,
\begin{align*}
 J(u) = J(u,\W,\l_+,\l_-) = \int_\W |Du|^2 + \l_+\chi_{\{u>0\}} + \l_-\chi_{\{u<0\}}
\end{align*}

The same proof given to show existence of minimizers of $J$ with a given boundary data also applies to our case. Here is the proposition and its proof can be adapted from the one in \cite{ACF84}.

\begin{proposition}
Given $g \in H^1(\W)$ such that $J(g) < \8$ there exists a minimizer $u \in H^1(\W)$ of $J$ such that $u - g \in H_0^1(\W)$.
\end{proposition}

Given $u$ a minimizer of $J$, over the domain $\cC_1 \sim B_1 \cap (\cR^2/\{\theta \in l\Z\})$, the pull back $\tilde f = f\circ\phi_l^{-1}$ is also a minimizer of $J$ over any compact set $\tilde K = \phi_l^{-1}(K)$. Most of the observations that can be made about the minimization problem posed in a domain in $\R^2$ can also be made about domains of the cone. Next we recall some of them.

First of all, since the functional $J$ is not convex, minimizers are not necessarily unique.

There is the possibility, when the boundary data is large enough, that minimizers stay positive in the whole domain and therefore the Euler Lagrange equations say that the solution has to be harmonic. Notice that in such case the regularity of the solution degenerates as $l$ becomes larger (see Proposition \ref{prop:harmonic_functions}). More interesting cases arise when there is a phase transition. This occurs, for example, if the boundary data changes sign or if it is sufficiently small.

The Euler Lagrange equation associated with the minimization problem looks exactly the same at every point of the domain which is different from the vertex. We have to introduce some notation before giving the set of equations. Let $u^\pm$ the positive and negative parts of $u = u^+ - u^-$ and $\W^\pm = \W\cap\{u^\pm>0\}$. Then
\begin{alignat*}{2}
 \D u &= 0 &&\text{ in $\{u \neq 0\}\cap (\W \sm \{0\})$},\\
 |Du^+| &= \l_+ &&\text{ in $(\p\W^+ \sm \p\W^-) \cap (\W \sm \{0\})$},\\
 |Du^-| &= \l_- &&\text{ in $(\p\W^- \sm \p\W^+) \cap (\W \sm \{0\})$},\\
 |Du^+|^2 - |Du^-|^2 &= \l_+^2 - \l_-^2 &&\text{ in $(\p\W^+ \cap \p\W^-) \cap (\W \sm \{0\})$}.\\
\end{alignat*}
What happens at the origin is actually the main concern of this work. Something that we can say is that if $u(0) \neq 0$ then $u$ is also harmonic at the origin. The next interesting case is when $0 \in (\p\W^+ \cup \p\W^-)$.


\subsection{Further properties of minimizers}

In this section we comment on some of the fundamental properties of minimizers of $J$. Their proof are simple adaptations of the classical proofs given in \cite{MR1759450,ACF84} and we leave them for the appendix of this paper. Specifically we will discuss:
\begin{enumerate}
\item Initial regularity. For any cone we show that minimizers are at least H\"older continuos depending on $l$ and the $H^1$ norm of the minimizer.
\item Stability of minimizers by uniform convergence.
\item Optimal regularity when $l\leq2\pi$.
\item Compactness and 1-homogeneity of sequences of blow-ups when $l\leq2\pi$.
\end{enumerate}


\subsubsection{Initial regularity}

Initially we can use the results from \cite{ACF84} to say that the minimizer $u$ is $C^{0,1}$ in every compact $\tilde K$ of the form $\tilde K = \phi_l^{-1}(K)$ and therefore also locally in $\W\sm\{0\}$. In particular, the Lipschitz estimates in \cite{ACF84} are scale invariant and therefore in our situation it gives us a Lipschitz estimate that degenerates towards the vertex.

\begin{proposition}\label{prop:Lipschitz_degenerate}
Given $u$ be a minimizer of $J = J(\cC_1,\l_+,\l_-)$ with $\|Du\|_{L^2(\cC_1)} \leq 1$ then for $r\in(0,1/2)$,
\begin{align*}
\|Du\|_{L^\8(\cC_{1/2}\sm\cC_r)} &\leq Cr^{-1},
\end{align*}
for some universal $C>0$.
\end{proposition}

The next step is to check that $u$ also remains continuous up to the vertex. In this sense we can show the following Theorem.

\begin{theorem}\label{thm:initial_reg}
Let $u$ be a minimizer of $J = J(\cC_1,\l_+,\l_-)$ with $\|Du\|_{L^2(\cC_1)} \leq 1$ then for any $\a \in (0,\min(1,2\pi/l))$ we have that $u \in C^{\a/4}(\cC_{1/10})$ with,
\begin{align*}
|u(x) - u(y)| \leq C|x-y|^{\a/4} \text{ for every $x,y\in\cC_{1/10}$},
\end{align*}
and some universal $C > 0$.
\end{theorem}

\begin{corollary}\label{coro:subharmonic}
Let $u$ be a minimizer of $J = J(\cC_1,\l_+,\l_-)$, then $u^\pm$ are continuous subharmonic functions satisfying $\D u^\pm = 0$ in $\cC_1^\pm$.
\end{corollary}


\subsubsection{Stability}

In the previous part we saw that for a minimizer $u$, the positive and negative parts $u^\pm$ are automatically subharmonic continuous functions and we even have a modulus of continuity for them. The stability of minimizers by uniform convergence depends on uniform equicontinuity and non degeneracy estimates. This allow us to say that if two minimizers are uniformly close then their zero sets are also close in the Hausdorff metric.

\begin{theorem}[Stability]\label{thm:stability}
Let $\{u_k\}$ be a sequence of minimizers of $J = J(\cC_1,\l_+,\l_-)$ with $\l_+$ and $\l_-$ different from zero converging to a function $u$ in $\cC_1$ with respect to the $H^1$ norm. Then:
\begin{enumerate}
 \item $\{u_k\}$ also converges uniformly to $u$ in $\cC_{1/2}$,
 \item Each one of the sets $\{u_k>0\}\cap\cC_{1/2}$ and $\{u_k<0\}\cap\cC_{1/2}$ converge to the respective set $\{u>0\}\cap\cC_{1/2}$, $\{u<0\}\cap\cC_{1/2}$ with respect to the Hausdorff distance,
 \item $u$ is also a minimizer of $J$.
\end{enumerate}
\end{theorem}
 

\subsubsection{Optimal regularity}

The optimal regularity expected for this problem can not be better than Lipschitz as in the classical case. On the other hand harmonic functions defined over cones with length $l > 2\pi$ may not be Lipschitz. Here we focus mainly on the case when $l\leq 2\pi$ in order to obtain the optimal regularity for the minimizers of $J$.

\begin{theorem}[Optimal regularity when $l \leq 2\pi$]\label{thm:optimal_reg_accute}
Let $l \leq 2\pi$, $u$ be a minimizer of $J = J(\cC_1,\l_+,\l_-)$ with $\|Du\|_{L^2(\cC_1)} \leq 1$ and $\{u = 0\}\cap\cC_{1/2} \neq \emptyset$; then for every $x_0 \in \cC_{1/4}^+$,
\begin{align*}
|Du(x_0)| &\leq C,\\
|u(x_0)| &\leq C\dist(x_0, \p\cC_1^+\cap\cC_{1/2}).
\end{align*}
\end{theorem}

In the case $l > 2\pi$ we can still can ask ourselves if the minimizer $u$ remains Lipschitz up to the vertex if $u(0)=0$. This is the case for instance of problems with one phase. This follows from the observation that away from the origin the problem inherits the regularity from the classical case, therefore the gradient along the free boundary is constant independently of how close we get to the origin. In the case of having two phases there might be still some balance between the positive and negative phase that allows the gradient to grow to infinity as we approach the vertex. However we suspect that when $\lambda^+ \neq \lambda^-$ this is not the case. 


\subsubsection{Blows-up}

As a consequence of the stability and the optimal regularity we obtain that a sequence of Lipschitz dilations of a given minimizer of $J$ and centered at the origin, have an accumulation point which is also a minimizer $J$ over any compact set of the cone. Moreover, by proving a monotonicity formula as in \cite{MR1759450} we obtain that such an accumulation point is a 1-homogeneous function.

\begin{corollary}[Blow-up limits]\label{coro:blowup} 
Let $l\leq 2\pi$ and $u$ be a minimizer of $J = J(\cC_1,\l_+,\l_-)$ with $u(0)=0$ and $\|Du\|_{L^2(\cC_1)} \leq 1$. For any sequence of blow-up $u_k = r_k^{-1}u(r_k\cdot)$ with $r_k\to0$ we have that there exist an accumulation point $u\in C^{0,1}_{loc}(\cC)$ such that:
\begin{enumerate}
\item $u$ is also a minimizer of $J(K,\l_+,\l_-)$ for any compact set $K\ss\cC$,
\item $u$ is a 1-homogeneous function in $\cR^2/\{\vec{\theta} \in l\Z\}$.  
\end{enumerate}
\end{corollary}

\section{The vertex and the free boundary: Case $l < 2\pi$}

In this section we show that if $\C$ is a cone with length $l < 2 \pi$, then $0 \notin (\p\W^+\cup\p\W^-)$ for any minimizer $u$. The idea is to reduce the problem to 1-homogeneous minimizers by using Corollary \ref{coro:blowup}.
\begin{theorem}\label{thm:avoid_vertex}
Let $l< 2\pi$ and $u$ be a minimizer of $J = J(\cC_1,\l_+,\l_-)$. Then $0 \notin(\p\W^+\cup\p\W^-)$.   
\end{theorem}

We split the proof into several Lemma.

\begin{lemma}\label{lemma:case1}
Let $l< 2\pi$ and $u$ be a minimizer of $J = J(\cC_1,\l_+,\l_-)$. Then $0 \notin (\p\W^+\cap\p\W^-)$. 
\end{lemma}

\begin{proof}
Suppose by way of contradiction that $0 \in (\p\W^+\cap\p\W^-)$. By Corollary \ref{coro:blowup} we have that there exists a limiting blow up $u_0$ which is homogeneous of order one. But homogeneous harmonic functions of order one are linear and then $\cH^1(\{u_0>0\}\cap\p B_1) = \cH^1(\{u_0<0\}\cap\p B_1)=\pi$. This is a contradiction with $l< 2\pi$. 
\end{proof}

\begin{remark} 
Lemma \ref{lemma:case1} coupled with the compactness and stability results from Section 2 works to show that there exists some $\e>0$ such that the same result holds for $l < \pi - \e$. We won't discuss this proof here as this result is contained in the following Lemmas.
\end{remark}

The previous Lemma reduces the problem to study only cases with just one phase. From now on we will assume without lost of generality that $\l_+=1, \l_-=0$ and the minimizers are non negative. Also, from the previous blow-up argument applied now to solutions with just one phase we can reduce the problem to showing that the function $v = x_2^+$ is not  a minimizer of $J(K)$ for any compact set $K\ss\cC$.

When we talk about the function $x_2^+$ defined in $\cC$ we mean the following: Because $l<2\pi$ there is an isometry
\begin{align*}
\phi:\cC\sm\{0\}\to\W = \{(x_1,x_2)\in\R^2: -\cot(l/2)|x_1| < x_2\}.
\end{align*}
It is in this coordinate system that we define the function $x_2^+$. In the next section we will use that for $u:K\subset\subset\cC\to\R$, the functional $J(u,K)$ can also be computed from $\tilde u = u\circ\phi^{-1}$ and $\tilde K = \phi^{-1}(K)$ in the following way,
\begin{align*}
J(u,K) = \tilde J(\tilde u,\tilde K) = \int_{\tilde K} |D\tilde u|^2 + |\{\tilde u>0\}\cap\tilde K|.
\end{align*}

Notice that a competitor $v_0$ for $v$ in $K$ such that $\{v_0 > 0\}\cap K \ss \{v>0\}\cap K$ gives that $\tilde J(\tilde v_0,\tilde K) > \tilde J(v,\tilde K)$ because $v$ is the unique minimizer of $\tilde J(\tilde K)$ with its boundary data. Therefore, if we want to find competitor with smaller values of $\tilde J(\tilde v_0,\tilde K)$ it is reasonable to look for competitors that add some positivity set to the positivity set that $v$ already has. This is the motivation for the following sections.

From now on we will drop the tildes and work exclusively in $\W = \{(x_1,x_2)\in\R^2: -\cot(l/2)|x_1| < x_2\}$.


\subsection{Reduction to a different optimization}

To find a better competitor than $v = x_2^+$ we will construct a bounded set $E \ss \R^2_- = \{x_2 < 0\}$, with Lipschitz boundary, such that the following expression is arbitrarily small meanwhile keeping the size of $E\cap(\R^2\sm\W)$ not too small,
\begin{align*}
F(E) = |E| - \int_\R u_E(x_1,0)dx_1,
\end{align*}
where $u_E$ is the solution of
\begin{alignat*}{2}
\Delta u_E &= 0 &&\text{ in } E \cup \R^2_+,\\
u_E &= x_2^- &&\text{ in } \R^2 \sm (E \cup \R^2_+),
\end{alignat*}
such that $u_E \to 0$ as $|x|\to\8$. 

Some properties of $F$ are given by the following Lemma.

\begin{lemma}\label{lemma:F}
Given $E \ss \R^2_-$ bounded and with Lipschitz boundary, we have that the following hold,
\begin{enumerate}
 \item Scaling: For $t>0$ the scaled set $tE$ satisfies $F(tE) = t^2F(E)$.
 \item Relation with $J$: For $u_{E,R}$ the solution of
\begin{alignat*}{2}
\Delta u_{E,R} &= 0 &&\text{ in } E_R \cup B_R^+,\\
u_{E,R} &= x_2^- &&\text{ in } \R^2 \sm (E \cup B_R^+),
\end{alignat*}
Then for $v_{E,R} = u_{E,R} + x_2$,
 \begin{align*}
 F(E) = \lim_{R\to\8} \1J(v_{E,R},B_R) - J(v,B_R)\2 \geq 0.
 \end{align*}
\end{enumerate}
\end{lemma}

\begin{proof}
(1) follows by the change of variables formula because $u_{tE}=tu_E(t^{-1}\cdot)$.

To prove (2) we take first $R$ sufficiently large such that $B_R \supseteq E$ and use that $v$ minimizes $J(B_R)$ while $v_{E,R}$ is harmonic in $E \cup B_R^+$,
\begin{align*}
0 &\leq J(v_{E,R},B_R) - J(v,B_R),\\
&= |E| + \int_{E \cup B_R^+} |Dv_{E,R}|^2 - |Dv|^2,\\
&= |E| - \int_{E \cup B_R^+} |D(v_{E,R}-v)|^2.
\end{align*}
We use now that in $E \cup B_R^+$ the following holds in the distributional sense $\D(v_{E,R}-v) = -\D v = -\chi_{\{x_2=0\}}\cH^1$.
\begin{align*}
0 &\leq J(v_{E,R},B_R) - J(v,B_R),\\
&= |E| - \int_\R (v_{E,R}-v)(x_1,0)dx_1,\\
&= |E| - \int_\R u_{E,R}(x_1,0)dx_1.
\end{align*}
Sending $R\to\8$ makes $u_{E,R}\to u_E$ uniformly in the bounded set $\bar E\cap\{x_2=0\}$ and therefore also in $\{x_2=0\}$ because both functions are zero in $\{x_2=0\}\sm(\bar E\cap\{x_2=0\})$. This implies that the integral of $u_{E,R}(\cdot,0)$ converges to the integral of $u_E(\cdot,0)$ and this concludes the Lemma.
\end{proof}

\begin{remark}\label{rmk:strategy}
The previous proof also works to show that,
 \begin{align*}
 F(E) - |E\cap(\R^2\sm\W)| \geq \lim_{R\to\8} \1J(v_{E,R},B_R\cap\W) - J(v,B_R\cap\W)\2.
 \end{align*}
 We just have to notice that,
 \begin{align*}
  &J(v_{E,R},B_R\cap\W) - J(v,B_R\cap\W)\\
  &\leq |E| - |E\cap(\R^2\sm\W)| + \int_{E \cup B_R^+} |Dv_{E,R}|^2 - |Dv|^2.
 \end{align*}
 In this sense we can make clear what is our strategy. By finding $E$ such that $F(E) - |E\cap(\R^2\sm\W)|<0$ we would be able to get a better competitor than $v$ in $B_R\cap\W$ for some $R$ sufficiently large.
\end{remark}


\subsection{Initial step}

The following Lemma gives an estimate of $F$ in isosceles triangles. This will be the basic configuration which we will use in our inductive construction.

\begin{lemma}\label{lemma:initial}
Given $c>0$, let $A_c$ the isosceles triangle with vertices $(-c,0),(c,0)$ and $(-1,0)$. Then $F(A_c) \leq 2$.
\end{lemma}

\begin{proof}
Let $B_{c,h}$ be the quadrileteral with vertices at $(-c,0),(0,h),(c,0)$ and $(-1,0)$. We construct first a function $w_{c,h}$ such that $J(w_{c,h},B_{c,h})-J(v,B_{c,h}) \leq 2$. Let for $(x_1,x_2)\in B_{c,h}\cap\{x_1\leq0\}$,
\begin{align*}
 w_{c,h}(x_1,x_2) = \frac{h}{c(h+1)}(x_1 + cx_2 + c).
\end{align*}
For $(x_1,x_2)\in B_{c,h}\cap\{x_1\geq0\}$ we define $w_{c,h}$ by extending it symmetrically, $w_{c,h}(x_1,x_2) = w_{c,h}(-x_1,x_2)$. Outside of $B_{c,h}$ we just make $w_{c,h} = v$. Notice that $w_{c,h}$ is continuous across $\p B_{c,h}$ and it is an admissible competitor against $v$ in any ball $B_R \supseteq B_{c,h}$.

\begin{figure}[h]
  \includegraphics[width = 10cm]{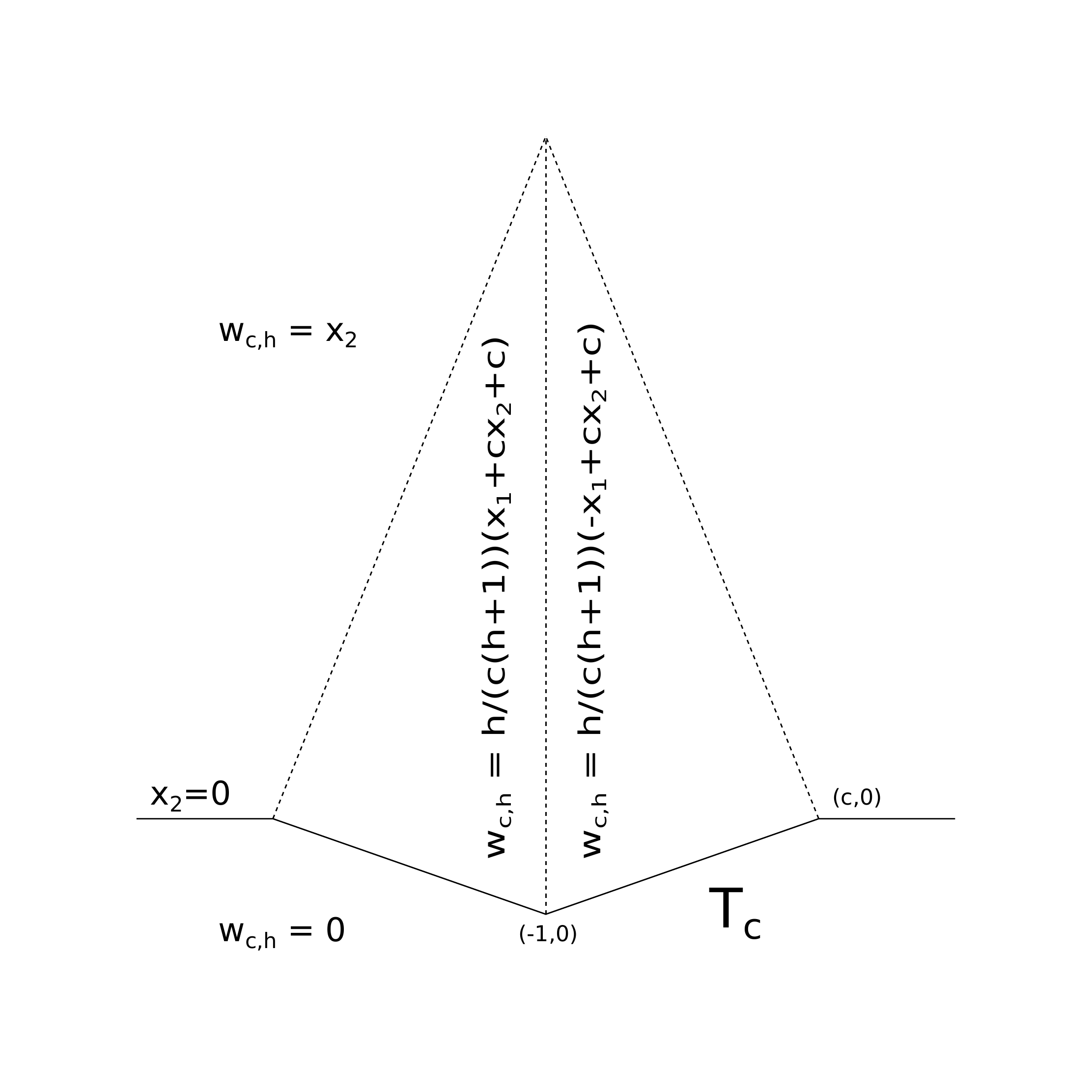}
\end{figure}

Let's compute the difference of the energies and then fix $h$ so that it minimizes it,
\begin{align*}
 J(w_{c,h},B_R) - J(v,B_R) &= c + \1\frac{h^2}{h+1}\2\1\frac{c^2+1}{c}\2 - ch.
\end{align*}
In order to minimize the previous expresion we chose $h = \sqrt{c^2+1}-1$. The previous difference is now,
\begin{align*}
 J(w_{c,h},B_R) - J(v,B_R) &= 2\frac{\sqrt{c^2+1}-1}{c} \leq 2.
\end{align*}

Now we replace $w_{c,h}$ by the harmonic function $v_{A_c,R}$ in $A_c\cup B_R^+$ taking the boundary values $v_R = w_{c,h} = v$ in $\p(A_c\cup B_R^+)$. This makes $J(v_{A_c,R},B_R) \leq J(w_{c,h},B_R)$ and $J(v_{A_c,R},B_R) - J(v,B_R) \leq 2$. By taking $R\to\8$ and using Lemma \ref{lemma:F} we obtain desired estimate for $F(A_c)$.
\end{proof}


\subsection{Inductive step}

Now we describe how to diminish the value of $F(E)$ inductively meanwhile keeping $|E\cap(\R^2\sm\W)|$ bounded away from zero. Consider a set $E \subset\subset \R\times[-1,0]$ and scale it by a factor $t \in (0,1)$, this diminishes the value of $F$ by a factor $t^2$. The next step is to translate $tE \cup \bar \R^2_+$ downwards a distance $(1-t)$ giving us,
\begin{align*}
E_t = \1\1tE \cup \bar\R^2_+\2 - (1-t)e_2\2 \cap \R^2_- \ss \R\times[-1,0].
\end{align*}
This set however is unbounded, so we truncate it by the trapezoid $T_{t,a,b}$, for $a>b>0$, with vertices at $(-a,0),(a,0),(-b,-(1-t))$ and $(b,-(1-t))$, obtaining in this way,
\begin{align*}
E_{t,a,b} = E_t \cap T_{t,a,b} \subset\subset \R\times[-1,0].
\end{align*}

\begin{figure}[h]
  \includegraphics[width = 9cm]{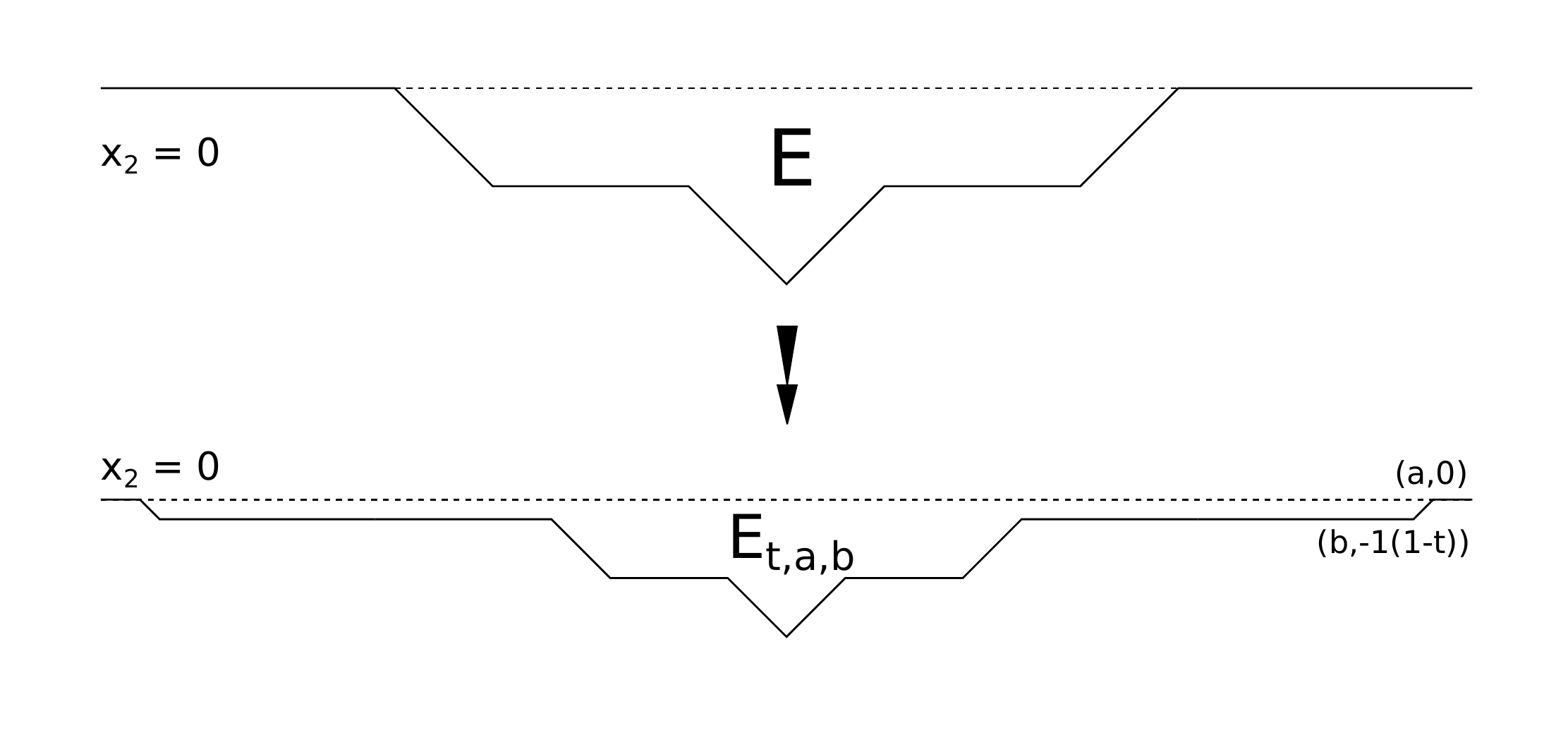}
\end{figure}

Formally we expect $F(E_t)$ to be $t^2F(E)$ however here we are actually subtracting two infinite quantities. The intuition behinds this is that the downwards translation of $tE$ adds as much volume as the amount in which the integral increases. We will see then that the truncation given by $T_{t,a,b}$ can be made such that it does not add to much to the functional. This is the motivation for the following Lemma.

\begin{lemma}\label{lemma:inductive}
Given $E \subset\subset \R\times[-1,0]$ with Lipschitz boundary and symmetric with respect to $\{x_1=0\}$ and $t\in(0,1)$ there exists $a_0>b_0>0$ sufficiently large such that $F(E_{t,a,b}) \leq t^2F(E) + 3(1-t)^2$ for any $a>\min(a_0,b)$ and $b>b_0$.
\end{lemma}

\begin{remark}
In the previous Lemma the optimal choice of $t$ in order to minimize the upper bound for $F(E_{t,a,b})$ is
\begin{align*}
t = \frac{3}{3+F(E)}
\end{align*}
for which
\begin{align*}
F(E_{t,a,b}) \leq \frac{3F(E)}{3+F(E)}.
\end{align*}
\end{remark}

\begin{proof}
We rewrite $F(E_{t,a,b})$ in the following way,
\begin{align*}
F(E_{t,a,b}) &= |tE| + (a+b)(1-t) - \int_\R u_{E_{t,a,b}}(x_1,0)dx_1,\\
&= |tE| - \int_{-b}^b\1u_{E_{t,a,b}}(x_1,0)-(1-t)\2dx_1,\\
&{} + \1(a-b)(1-t) - 2\int_b^a u_{E_{t,a,b}}(x_1,0)dx_1\2.
\end{align*}

Now we compare $(u_{E_{t,a,b}}-(1-t))$ with $\tilde u_t = u_{tE}(\cdot+(1-t)e_2)$ in order to include $F(tE) = t^2F(E)$ in the right hand side. $\tilde u_t$ satisfies,
\begin{alignat*}{2}
\Delta \tilde u_t &= 0 &&\text{ in } \1tE \cup \bar\R^2_+\2 - (1-t)e_2,\\
\tilde u_t &= x_2^- - (1-t) &&\text{ in } \R^2 \sm \1\1tE \cup \bar\R^2_+\2 - (1-t)e_2\2.
\end{alignat*}
Similarly $(u_{E_{t,a,b}}-(1-t))$ satisfies,
\begin{alignat*}{2}
\Delta (u_{E_{t,a,b}}-(1-t)) &= 0 &&\text{ in } E_{t,a,b}\cup \bar\R^2_+,\\
(u_{E_{t,a,b}}-(1-t)) &= x_2^- - (1-t) &&\text{ in } \R^2 \sm \1E_{t,a,b}\cup \bar\R^2_+\2.
\end{alignat*}
Notice that sending $a\to\8$ makes the domain $E_{t,a,b}\cup \bar\R^2_+$ to approach the domain $\1tE \cup \bar\R^2_+\2 - (1-t)e_2$ locally with respect to the Hausdorff distance. This implies that as $a\to\8$ we have that $(u_{E_{t,a,b}}-(1-t))\to\tilde u_t$ locally uniformly. Given $\e>0$, there is some $a$ sufficiently large such that,
\begin{align*}
\int_{-b}^b\1u_{E_{t,a,b}}(x_1,0)-(1-t)\2dx_1 &\leq \int_{-b}^b\tilde u_t(x_1,0)dx_1 + \e,\\
&= \int_{-b}^b u_{tE}(x_1,1-t)dx_1 + \e,
\end{align*}
We can then chose $b$ sufficiently large such that, by using the Poison kernel of the half plane,
\begin{align*}
\int_{-b}^b\1u_{E_{t,a,b}}(x_1,0)-(1-t)\2dx_1 &\leq \int_{\R}u_{tE}(x_1,1-t)dx_1 + 2\e,\\
&= \int_{\R} u_{tE}(x_1,0)dx_1 + 2\e.
\end{align*}
Giving us the following comparison between $F(E_{t,a,b})$ and $F(tE)$ for $a$ and $b$ sufficiently large,
\begin{align*}
F(E_{t,a,b}) &\leq F(tE) + 2\e + \1(a-b)(1-t) - 2\int_b^a u_{E_{t,a,b}}(x_1,0)dx_1\2.
\end{align*}
We will se now that the last term is controlled by $2(1-t)^2$. Then we will set $2\e = (1-t)^2$ to conclude the Lemma.

Let $c = (1-t)^{-1}(a-b)$ and $u_{(1-t)A_c}$ where the triangle $A_c$ is the same from Lemma \ref{lemma:initial}. We have the inclusion $(1-t)A_c + be_1 \ss E_{t,a,b}$ which implies that $u_{(1-t)A_c}(\cdot - be_1) \leq u_{E_{t,a,b}}$ and then,
\begin{align*}
2\int_b^a u_{E_{t,a,b}}(x_1,0)dx_1 &\geq 2\int_b^a u_{(1-t)A_c}(x_1 - b,0)dx_1,\\
&= \int_{\R} u_{(1-t)A_c}(x_1,0)dx_1,\\
&= |(1-t)A_c| - F((1-t)A_c),\\
&= (1-t)(a-b) - 2(1-t)^2.
\end{align*}
Therefore,
\begin{align*}
(a-b)(1-t) - 2\int_b^a u_{E_{t,a,b}}(x_1,0)dx_1 \leq 2(1-t)^2.
\end{align*}
Which is what we were looking for.
\end{proof}


\subsection{Proof of Theorem \ref{thm:avoid_vertex}}

The following Lemma combined with the previous Lemma \ref{lemma:case1} will complete the proof of Theorem \ref{thm:avoid_vertex}.

\begin{lemma}
Given $l<2\pi$, there exits a set $E$ and a radius $R$ sufficiently large such that $J(v_{E,R},B_R\cap\W) < J(v,B_R\cap\W)$.
\end{lemma}

\begin{proof}
Let $A_c$ the isocales triangle described in Lemma \ref{lemma:initial}, let $D = A_c\cap(\R^2\sm\W)\cap\{x_2\geq-3/5\}$ and try to find $E$ such that:
\begin{enumerate}
 \item $D \ss E$,
 \item $F(E) < |D|$.
\end{enumerate}
By having this we use the Remark \ref{rmk:strategy} which says that
\begin{align*}
0 > F(E) - |E\cap(\R^2\sm\W)| \geq \lim_{R\to\8} \1J(v_{E,R},B_R\cap\W) - J(v,B_R\cap\W)\2
\end{align*}
and implies the Lemma.

Let $E_0 = A_c$, we know that,
\begin{enumerate}
 \item $D \ss E_0$
 \item $F(E_0)\leq 2$ from Lemma \ref{lemma:initial}. 
\end{enumerate}
Given $E_k$ let,
\begin{align*}
 F_k &= F(E_k),\\
 t_k &= \frac{3}{3+F_k},\\
 E_{k+1} &= (E_k)_{t_k,a_k,b_k}.
\end{align*}
with $a_k$ and $b_k$ sufficiently large such that Lemma \ref{lemma:inductive} applies and
\begin{align*}
F_{k+1} \leq \frac{3F_k}{3+F_k}.
\end{align*}
It is easy to show that such recurrence relation makes $F_k\to0$ as $k\to\8$. Eventually there will be some $k_0$ sufficiently large such that $F_{k_0} \leq |D|$. We now note that $F_{k_0} \leq |D|$ independently of how large $c$ was chosen in constructing $A_c$. $k_0$ will only depend on the length $l$ of the cone. Since we need to apply the iteration only $k_0$ times, we may choose $c$ large enough in the construction of $A_c=E_0$ so that 
\[
D \subset E_{k_0}
\]
Then we just have to chose $E = E_{k_0}$ to conclude the Lemma. 
\end{proof}


\subsection{Stability}

When we combine Theorem \ref{thm:avoid_vertex} with the stability given by Theorem \ref{thm:stability} we are able to say that the vertex not only is not in the free boundary but stays away from it a given distance.

\begin{corollary}
Let $l< 2\pi$ and $u$ be a minimizer of $J = J(\cC_1,\l_+,\l_-)$ with $\|Du\|_{L^2(\cC_1)} \leq 1$ then there exists some $\e = \e(l) > 0$ such that $\cC_\e \cap (\p\W^+\cup\p\W^-) = \emptyset$.
\end{corollary}

\begin{proof}
Proceed by contradiction assuming that there exists a sequence of minimizers $\{u_k\}_{k\in\N}$ with $\|Du\|_{L^2(\cC_1)} \leq 1$ such that for $\W^+ = \{u_k > 0\}\cap\cC_1$ and $\W^+ = \{u_k < 0\}\cap\cC_1$ we have that
\begin{align*}
\cC_{1/k} \cap (\p\W^+_k\cup\p\W^-_k) \neq \emptyset.
\end{align*}
By Theorem \ref{thm:initial_reg} the sequence is equicontinuous and also bounded therefore by Arzela-Ascoli it has a subsequence which converges uniformly to some function $u_0$ such that $0 \in (\p\W^+_0\cup\p\W^-_0)$, with $\W^\pm$ defined in a similar way. By the stability given \ref{thm:stability} we know that $u_0$ is a minimizer too but this contradicts Theorem \ref{thm:avoid_vertex}.
\end{proof}

\section{The vertex and the free boundary: Case $l \geq 2\pi$}

In this section we discuss the problem of determining whether the vertex may belong to the free boundary in the case $l \geq 2 \pi$. We show some examples when $0 \in (\p\W^+\cup\p\W^-)$ using the well known fact that when $l = 2 \pi$ and $\l_+ = 1,\l_-=0$ then $u=x^+_2$ is a minimizer of this type. Moreover it is the unique minimizer of $J(K)$ for any compact set $K\ss\R^2$ subject to its own boundary values.

\subsection{One phase free boundary through the vertex} \label{ss:pass}

Consider $l \geq 2 \pi$, $\l_+ = 1,\l_-=0$. The function $u=x^+_2$ defined in $\R^2\sm\{x_1=0,x_2<0\}$ can also be considered in $\cC$ by using an isometry $\phi:\R^2\sm\{x_1=0,x_2<0\}\to U \ss \cC$. Even though $\phi$ is not an isometry between $\R^2\sm\{x_1=0,x_2<0\}$ and $\cC_1$, we can consider $\tilde u = u\circ\phi:U \to \R$ and then extend it to $\cC$ by making it zero in $\cC\sm U$. We will drop now the tilde and consider $u=x^+_2$ defined in $\cC$. 

Let $v$ be a function on $\cC_1$ such that it has the same boundary values as $u$ in $\cC_1$ and minimizes $J(\cC_1)$. We will show that $v \equiv u$. In the quotient $\cR^2/\{\theta\in l\Z\}$ and after an appropriated rotation $u$ can be considered as $u(r,\theta) = r\cos\theta$. It satisfies that $u(r,\theta) = u(r,-\theta)$. Let now $\tilde v$ be defined by $\tilde v(r,\theta) = v(r,-\theta)$. Both functions $v$ and $\tilde v$ have the same boundary values a $u$ in $\cC_1$ and also $\tilde v$ minimizes $J(\cC_1)$. Consider now $v^+ = \max(v,\tilde v)$ and $v^- = \min(v,\tilde v)$. By the lattice principle Lemma \ref{lattice} 
both $v^\pm$ are minimizers of $J(\cC_1)$ with the same boundary data and symmetry as $u$. 

At this point we see that $v^{\pm}\circ\phi^{-1}$ also minimizes $J(B_1)$ with the same boundary values as $u=x^+_2$. The symmetry across $\{x_1=0\}$ implies that the Dirichlet term does not add to the functional if we include the segment $\{x_1=0,x_2\in(0,-1)\}$. However $u=x^+_2$ was the unique minimizer to that problem and therefore $v^{\pm} \equiv u$, so $v \equiv u$. Going back to $\cC$ we have found a minimizer with $0 \in (\p\W^+\cup\p\W^-)$.

\subsection{More than one positive phase free boundary through the vertex}

The previous idea can be extended to construct examples where two positive phases meet at the vertex. This is something unexpected since in the case when $l = 2\pi$ we know that the free boundary is smooth.

Consider $l \geq 4 \pi$, $\l_+ = 1,\l_-=0$, $\cC$ parametrized by $\cR^2/\{\theta\in l\Z\}$ and two isometries,
\begin{align*}
\phi_+:U^+ = \{\theta\in(-\pi,\pi)\} \to \R^2\sm\{x_2=0,x_1<0\},\\
\phi_-:U^- = \{\theta\in(l/2-\pi,l/2+\pi)\} \to \R^2\sm\{x_2=0,x_1>0\}
\end{align*}

In this case the two functions $u_\pm = x_1^\pm$ can be pasted together to construct a function $u = u^+\circ\phi_+ + u^-\circ\phi_-$ such that $u = u^\pm\circ\phi_\pm$ in $U^\pm$. We now consider a competitor $v$. If $v$ is a minimizer, we may use the lattice principle as before so that we may assume symmetry for $v$ across the lines that would be horizontal and vertical in Figure \ref{f:paste}. By cutting along the vertical line, we may use each half of $v$ as a competitor against $x_1^+$ on the cone $\tilde{\cC_1}$ which has half the lenth of the cone $\cC_1$. If $J(v) \leq J(u)$ on $\cC$, then necessarily each half must minimize, so $J(v) \leq J(x_1^+)$ on $\tilde{\cC_1}$. As shown in Section \ref{ss:pass} above, $x_1^+$ is the unique minimizer subject to its own boundary values on $\tilde{\cC_1}$, so we conclude each half of $v$ is identical to $x_1^+$. 

\begin{figure}[h]
  \includegraphics[width = 8cm]{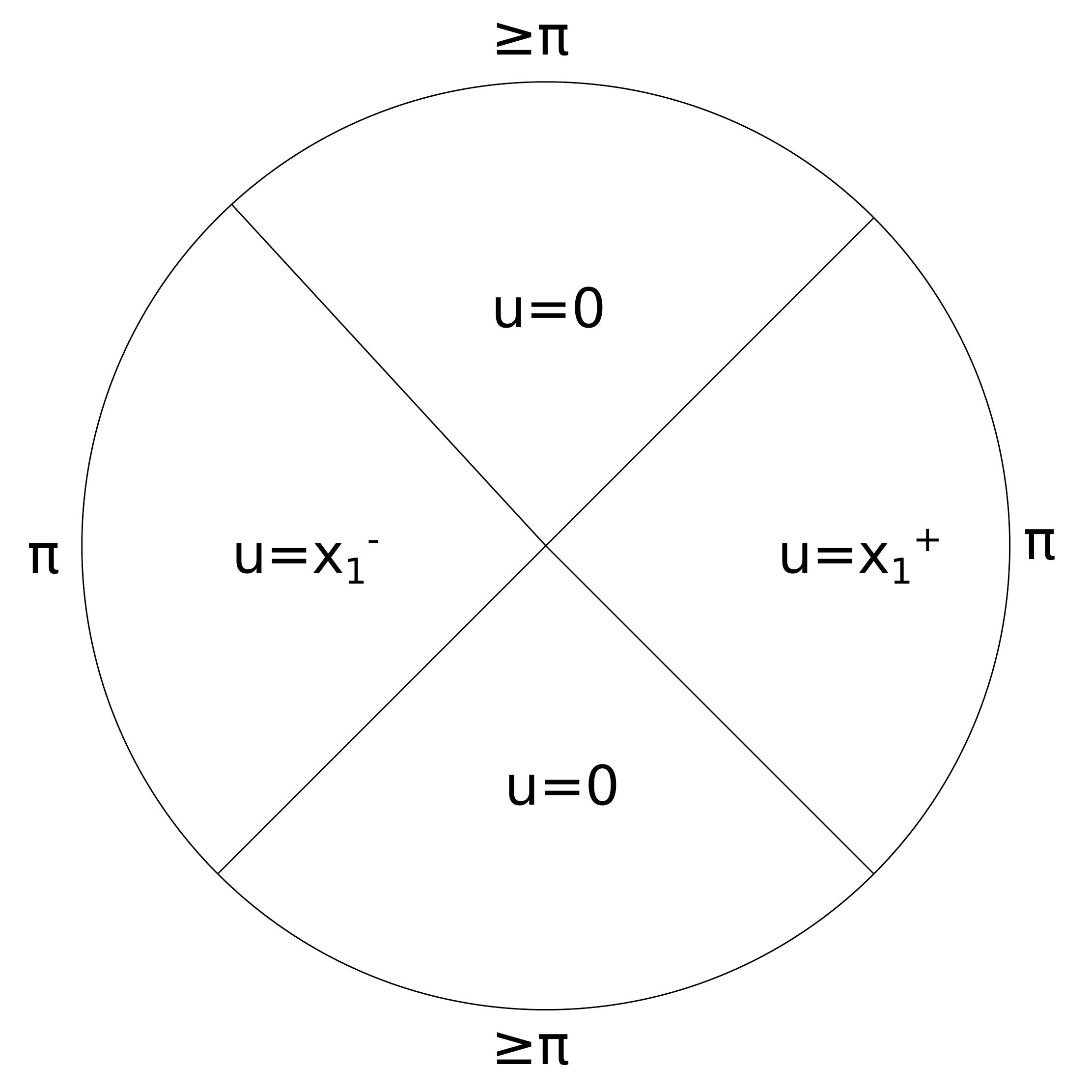}
  \caption{Pasting two linear pieces}
  \label{f:paste}
\end{figure}

This construction can also be generalized to show that $k$ phases can meet at the vertex if $l \geq 4k\pi$. 

\section{Appendix}


\subsection{Monotonicity formulas}

Monotonicity formulas for harmonic and subharmonic functions allow us to control infinitesimal quantities by integral ones. The classical monotonicity for the average of the Dirichlet energy of a harmonic function or the Alt-Caffarelli-Friedman (ACF) formula can be applied when the domain of integration doesn't contain the vertex. When we decide to center the integrals at the vertex then they are no longer valid and the classical proofs have to be slightly modified.

Given $r>0$, we fix $\cC_r$ to be the intersection of $\cC$ with the ball of radius $r$ centered at the origin.


\begin{lemma}[Monotonicity of the average Dirichlet energy]\label{lemma:dirichlet_monotonicity}
Let $u$ be a harmonic function over the cone $\cC$ with length $l$. Then
\begin{align*}
\frac{1}{r^{2\a}} D(\cC_r,u) = \frac{1}{r^{2\a}} \int_{\cC_r}|Du|^2
\end{align*}
is an increasing function of $r$ for $\a \in (0,2\pi/l]$
\end{lemma}

\begin{proof}
Integrating by parts,
\begin{align*}
\frac{1}{r^{2\a}} D(\cC_r) = \frac{1}{r^{2\a}} \int_{\p\cC_r}uu_rd(r\theta).
\end{align*}
Now we use the Fourier representation of $u$ and the fact that the sequence of functions given by the sines and cosines are and orthogonal set in $L^2(\cC_r)$. Let
\begin{align*}
u = \sum_{k=0}^{\infty}{r^{2 \pi k/l} \left(a_k \cos \frac{2 \pi k}{l} \theta 
                                                \ + \ b_k \sin \frac{2 \pi k}{l} \theta \right) },
\end{align*}
then
\begin{align*}
\frac{1}{r^{2\a}}\int_{\p\cC_r}uu_r &= \sum_{k=0}^{\infty} (2\pi k/l) r^{4 \pi k/l-2\a} \int_0^l a_k^2 \cos^2 \frac{2 \pi k}{l} \theta + b_k^2 \sin^2 \frac{2 \pi k}{l} \theta d\theta,\\
&= \pi\sum_{k=1}^{\infty} kr^{4 \pi k/l-2\a} (a_k^2+b_k^2).
\end{align*}
As $\a \in (0,2\pi/l]$, the exponents appearing on the sum above are all non negative, each term the is non decreasing in $r$ and the whole series is therefore non decreasing in $r$.
\end{proof}

\begin{remark}
At any other point $x_0 \neq 0$ we can also define the ball $B_r(x_0) \ss \cR^2$. As far as $r \leq |x_0|$ this ball looks exactly as the flat ball we are use to. In that case the monotonicity proof given above works with any exponent $2\a \in (0,1]$. Therefore we obtain for $\a \in (0,\min(1,2\pi/l)]$,
\begin{align*}
\frac{1}{r^\a}D(B_r(x_0),u) = \frac{1}{r^{2\a}}\int_{B_r(x_0)}|Du|^2,
\end{align*}
is also increasing with the restriction that $r \leq |x_0|$ if $x_0\neq 0$. 
\end{remark}


\begin{lemma}[Alt-Caffarelli-Friedman monotonicity formula] \label{lemma:ACF}
Let $\{u_+, u_-\}$ be a pair of nonnegative continuous subharmonic functions on the cone $\C_1$ with length $l \leq 2 \pi$ such that $ u_+ \cdot u_- =0$ in $\C_1$ and $\a \in (0,4\pi/l]$. Then the functional
\begin{align*}
r \mapsto \Phi(r,u_+,u_-) = \frac{1}{r^{2\a}} \int_{\C_r}{|D u_+|^2} \int_{\C_r}{|D u_-|^2}
\end{align*}
is nondecreasing for $0<r<1$.
\end{lemma}

\begin{proof}
As in the classical proof we have that,
\begin{align*}
\frac{r\Phi'(r)}{2\Phi(r)} &= -\a + \frac{\int_{\p \cC_1}|Du^+|^2}{2\int_{C_1}|Du^+|^2} + \frac{\int_{\p C_1}|Du^-|^2}{2\int_{C_1}|Du^-|^2},\\
&\geq -\a + \1\frac{\int_0^l(\tilde u^+_\theta)^2d\theta}{\int_0^l(\tilde u^+)^2d\theta}\2^{1/2} + \1\frac{\int_0^l(\tilde u^+_\theta)^2d\theta}{\int_0^l(\tilde u^-)^2d\theta}\2^{1/2}.
\end{align*}
The last two terms get minimized by the first eigenvalues of the support of $u^\pm$. They become even smaller if we assume that each one of these two domains are connected and have complementary lengths $m$ and $l-m$. In that case the eigenvalues are $-(\pi/m)^2$ and $-(\pi/(l-m))^2$. So the expression above gets minimized when $m=l/2$ and then,
\begin{align*}
\frac{r\Phi'(r)}{2\Phi(r)} &\geq -\a + 4\pi/l,
\end{align*}
which is non negative for $\a \in (0,4\pi/l]$.
\end{proof}


\subsection{Initial regularity}

We can get some regularity for the minimizer $u$ by just comparing it with its harmonic replacement in a given ball.

\begin{lemma}
Let $u$ be a minimizer of $J = J(\cC_1,\l_+,\l_-)$ with $\|Du\|_{L^2(\cC_1)} \leq 1$, then
\begin{align*}
|u(x) - u(0)| \leq C|x|^{\a/3} \text{ for every $x \in \cC_{1/10}$},
\end{align*}
for any $\a \in (0,\min(1,2\pi/l))$ and some universal $C = C(\a) > 0$.
\end{lemma}

\begin{proof}
We prove first that for every $x_0 \in B_{1/2}$, $r\in(0,1/2)$ and $\a \in (0,\min(1,2\pi/l))$,
\begin{align}\label{eq:Morrey1}
\int_{\cC_r}|Du|^2 \leq Cr^{2\a}.
\end{align}
Consider $0<r<R<1/2$ and $h_R$ the harmonic function in $\cC_R$ taking the same boundary values as $u$ in $\p\cC_R$. Then $v$ is an admissible competitor for $J$ against $u$ in $\cC_R$ from where we get,
\begin{align*}
\int_{\cC_R} |Du|^2 - |Dh_R|^2 \leq CR^2.
\end{align*}
Because $h_R$ is harmonic and $u-h_R \in H^1_0(\cC_R)$,
\begin{align*}
\int_{\cC_R} |D(u-h_R)|^2 = \int_{\cC_R} |Du|^2 - |Dh_R|^2 \leq CR^2.
\end{align*}
Now we estimate how much $D(u)$ grows from $r$ to $R$ from Lemma \ref{lemma:dirichlet_monotonicity} and the fact that $h_R$ minimizes the Dirichlet energy in $\cC_R$,
\begin{align*}
\int_{\cC_r}|Du|^2 &\leq \int_{\cC_R}|D(u-h_R)|^2 + \int_{\cC_r}|Dh_R|^2,\\
&\leq CR^2 + \1\frac{R}{r}\2^{4\pi/l}\int_{\cC_R}|Dh_R|^2,\\
&\leq CR^2 + \1\frac{R}{r}\2^{4\pi/l}\int_{\cC_R}|Du|^2.
\end{align*}
From this we conclude \eqref{eq:Morrey1} by applying Lemma 3.4 in \cite{Fanghua11}.

Now we proof, in a similar way as in the Morrey estimates, that for $R \in (0,1/2)$,
\begin{align}\label{eq:Morrey2}
\int_{\cC_R} \frac{|u(y)-u(0)|}{|y|}dy \leq CR^{1+\a}.
\end{align}
The following computations can be made rigorous after regularizing $u$ by a convolution. We obtain the slope of $u$ between $0$ and $y$ by performing the following integral,
\begin{align*}
\frac{|u(y)-u(0)|}{|y|} &\leq \int_0^1|Du(ty)|dt.
\end{align*}
Next we integrate in $\p\cC_{r}$,
\begin{align*}
\int_{\p\cC_r}\frac{|u(y)-u(0)|}{|y|} &\leq \int_0^1 \frac{dt}{t} \int_{\p\cC_{tr}}|Du|.
\end{align*}
Finally we integrate with respect to $r$ between $0$ and $R$, apply H\"older's inequality and the previous estimate \eqref{eq:Morrey1},
\begin{align*}
\int_{\cC_R} \frac{|u(y)-u(0)|}{|y|} &\leq \int_0^1 R^2dt \frac{1}{(tR)^2}\int_{\cC_{tR}} |Du|,\\
&\leq \int_0^1 R^2dt \1\frac{1}{(tR)^2}\int_{\cC_{tR}} |Du|^2\2^{1/2},\\
&\leq CR^{1+\a}\int_0^1 t^{-1+\a}dt.
\end{align*}
As $\a>0$ the integral above is finite and we conclude \eqref{eq:Morrey2}.

Finally we use \eqref{eq:Morrey2} to compare $u(0)$ with $u(x)$ with $x \in \cC_{1/10}$. Consider a parameter $\e\in(0,1/2)$ to be fixed and $r = \e|x|$, we apply first the triangular inequality and then integrate over $B_{\e r}(x)$ using Proposition \ref{prop:Lipschitz_degenerate},
\begin{align*}
|u(x) - u(0)||B_{\e r}(x)| &\leq \int_{B_{\e r}(x)}|u(y)-u(x)|dy + \int_{\cC_{2r}}|u(y)-u(0)|dy,\\
&\leq Cr^{-1}\int_{B_{\e r}(x)}|x-y|dy + 2r\int_{\cC_{2r}} \frac{|u(y)-u(0)|}{|y|}dy,\\
&\leq C\1\e^3 r^2 + r^{2+\a}\2,
\end{align*}
It implies that $|u(x) - u(0)| \leq C(\e + r^\a\e^{-2})$, then we just chose $\e = r^{\a/3} (\leq 10^{-1/3} < 1/2)$ to conclude the proof.
\end{proof}

Here is the proof of Theorem \ref{thm:initial_reg}. Notice that the estimate degenerates in two ways, as $l$ grows and also as the H\"older exponent goes to $\min(1,2\pi/l)$.

\begin{proof}[Proof of Theorem \ref{thm:initial_reg}]
Let $x,y \in \cC_{1/10}$ and assume without lost of generality that $y$ is closest one to $0$. Let $\e\in(0,1/2)$ a parameter to be fixed and $r = |x|$. We consider two cases according if $y$ belongs or not to $B_{\e r}(x)$.

If $y \in B_{\e r}(x)$. Then we use the Lipschitz estimate from Proposition \ref{prop:Lipschitz_degenerate} to get
\begin{align*}
|u(x)-u(y)| \leq Cr^{-1}|x-y|.
\end{align*}
Given that $\e \leq r^{\a/(4-\a)} (\leq 10^{-1/3} < 1/2)$ we obtain that $r^{-1}|x-y| \leq |x-y|^{\a/4}$.

If $y \notin B_{\e r}(x)$ then we use the previous Lemma to get that
\begin{align*}
|u(x)-u(y)| &\leq |u(x)-u(0)| + |u(y) - u(0)|,\\
&\leq Cr^{\a/3}.
\end{align*}
Given that $\e = r^{1/3} (\leq r^{\a/(4-\a)})$ we obtain $r^{\a/3} \leq (\e r)^{\a/4} \leq |x-y|^{\a/4}$.
\end{proof}


\subsection{Stability}

We start by proving a non degeneracy estimate at the vertex. As we have done before we will use the already known results for the flat metric case when the problem is considered away from the origin. At the origin we will a 
non degeneracy result. The following Lattice Principle will be used to obtain non degeneracy. 


\begin{lemma}[Lattice Principle]  \label{lattice}
Let $u,v$ be two minimizers on $\cC_r$ with $u \leq v$ on $\partial \cC_r$. Then 
$\underline{w} = \min \{u,v\}$ and $\overline{w} = \max \{u,v\}$ are minimizers on $\cC_r$ subject to their respective boundary conditions. 
\end{lemma}

\begin{proof}
One may easily check that
\[
J(\underline{w})+ J(\overline{w}) =J(u)+ J(v)
\]
Since $\underline{w}=u$ and $\overline{w}=v$ on $\partial \cC_r$ it follows that $\underline{w}$ and $\overline{w}$ are minimizers of $J$. 
\end{proof}

\begin{lemma}[Non degeneracy at the vertex]
Let $u$ be a minimizer of $J = J(\cC_1,\l_+,\l_-)$ with $\l_+$ and $\l_-$ different from zero. For $\e\in(0,1)$ there exists some $\d > 0$ such that $|u|<\d$ in $\cC_\e$ implies $u=0$ in $\cC_{\e/2}$.
\end{lemma}

\begin{proof}
Given $\d>0$ we will consider a competitor $\varphi_\d$ which minimizes $J$ with constant boundary value $\d$ in $\p\cC_\e$. By Lemma \ref{lattice} and Theorem \ref{thm:stability} we may take the $\sup$ of all minimizers and conclude there is a unique minimizer $\varphi_\d$ that lies above every other minimizer with constant boundary value $\d$. Any rotation of $\varphi_\d$ is again a minimizer, and so $\varphi_\d$ is a radially symmetric minimizer $\varphi_\d$.  Given that $\d$ is sufficiently small one may easily compute that 
\begin{align*}
 \varphi_\d(x) = \l_+r_0\ln^+(r_0/|x|),
\end{align*}
where $r_0$ is the largest of the two roots of $\l_+r_0\ln(r_0/\e) = \d$. In particular $\varphi_\d$ vanishes in $B_{\e/2}$ if we chose $\d$ small enough.

Assuming that $|u|<\d$ by Lemma \ref{lattice} we have that $v=\max(\varphi_\d,u)$ is a minimizer over $\cC_\e$ and as stated above $v=\max(\varphi_\d,u) \leq \varphi_\d$. Then $u$ vanishes in $B_{\e/2}$.  
\end{proof}

\begin{corollary}[Stability of the zero set]\label{coro:stability_of_the_zero_set}
Let $u_1$ and $u_2$ be minimizers of $J = J(\cC_1,\l_+,\l_-)$ with $\l_+$ and $\l_-$ different from zero. For any $\e\in(0,1/2)$ there exists some $\d > 0$ such that $|u_1-u_2| < \d$ implies $\{u_1=0\}\cap\cC_1$ and $\{u_2=0\}\cap\cC_1$ are $\e$-close in the Hausdorff distance.
\end{corollary}
 
\begin{proof}
We have to show that $\{u_1=0\}\cap\cC_1 \ss (\{u_2=0\}\cap\cC_1) \oplus B_\e$ and by interchanging the roles of $u_1$ and $u_2$ we would have concluded the corollary. If the vertex doesn't belong to $\{u_1=0\}\cap\cC_1$ then the result follows from the classical theory by isolating the vertex. So we will assume in this proof that $u_1(0)=0$. The idea is to use the compactness of $\{u_1=0\}\cap\bar\cC_1$ to put togheter the results away from the origin and at the origin.

For $x \in \{u_1=0\}\cap(\cC_1\sm\{0\})$ we can use the classical theory in a ball $B_{r(x)}(x)$ with $r(x)=\min(|x|,\e/2)$ to conclude that there is some $\d(x)>0$ such that if $|u_1-u_2|<\d(x)$ in $B_{r(x)}(x)$, then $\{u_2=0\} \cap B_{r(x)}(x) \neq \emptyset$. Notice however that $\d(x)$ degenerates as $x\to0$.

We use the previous Lemma in the vertex in following form. Assume with out lost of generality that $u_2(0)\in(0,\d_0)$ and lets see that $\{u_2=0\}\cap\cC_{\e/2}\neq\emptyset$ if we chose $\d_0$ sufficiently small. Assume by contradiction that $u_2$ is a harmonic positive function in $\cC_{\e/2}$. By Harnack's inequality $u_2 \in (0,C\d_0)$ in $\cC_{\e/4}$ and by having that $\d_0$ is small enough we obtain a contradiction with the previous Lemma.

Consider the covering of $\{u_1=0\}\cap\bar\cC_1$ given by $\{B_{r(x)}(x)\} \cup \cC_{\e/2}$ for $x$ ranging over $\{u_1=0\}\cap(\bar\cC_1\sm\{0\})$. Extract then a finite collection $x_1,\ldots,x_N$ such that $\cC_{\e/2},B_{r(x_1)}(x_1),\ldots,B_{r(x_N)}(x_N)$ still covers $\{u_1=0\}\cap\bar\cC_1$ and chose $\d$ to be the smallest number among $\d_0,\d(x_1),\ldots,\d(x_N)$. From the previous considerations we have that $\{u_2=0\}\cap\cC_1$ hits each one of the sets  $\cC_e,B_{r(x_1)}(x_1),\ldots,B_{r(x_N)}(x_N)$ which implies that for every $x \in \{u_1=0\}\cap\bar\cC_1$ there is some $y \in \{u_2=0\}\cap\cC_1$ such that $\dist(x,y)<\e$. This is equivalent to say that $\{u_1=0\}\cap\cC_1 \ss \{u_2=0\}\cap\cC_1 \oplus B_\e$ which concludes the proof.
\end{proof}

Here is the proof of Theorem \ref{thm:stability}

\begin{proof}[Proof of Theorem \ref{thm:stability}]
By the regularity already proved in Theorem \ref{thm:initial_reg} we have that the sequence is uniformly in $C^{\a/3}(\cC_{1/2})$. By Arzela-Ascoli we have that the sequence has an accumulation point $v \in C^{\a/3}(\cC_{1/2})$ with respect to the $C^{\beta}$ normn for $\beta < \alpha /3$. By having that $u_k\to u$ in $L^2(\cC_{1/2})$ we obtain that $u$ is the only possible accumulation point in $L^2(\cC_{1/2})$ and therefore $v=u$ and the whole sequence converges uniformly to $u$ which proves the first part. The second part follows now from Corollary \ref{coro:stability_of_the_zero_set}.

To conclude that $u$ is a minimizer of $J$ we use as in the classical proof the lower semicontinuity of the Dirichlet term and then the uniform convergence of $\{u_k>0\}\cap\cC_{1/2}$ and $\{u_k<0\}\cap\cC_{1/2}$ to $\{u>0\}\cap\cC_{1/2}$ and $\{u<0\}\cap\cC_{1/2}$ respectively.
\end{proof}


\subsection{Optimal regularity} 

The following Lemma and its Corollary gives a gradient bound at the free boundary points. Recall that for a set $\Omega$ we have defined $\Omega^+ = \Omega \cap \{u>0\}$ and $\Omega^-$ is defined similarly. 

\begin{lemma}\label{lemma:bound_gradient}
Let $l \leq 2\pi$, $u$ be a minimizer of $J = J(\cC_1,\l_+,\l_-)$ with $\|Du\|_{L^2(\cC_1)} \leq 1$ and let $x_0 \in (\p\W^+\cap\p\W^-)\cap (\cC_{1/2}\sm\{0\})$; then $|Du(x_0)| \leq C$ for some universal constant $C>0$.
\end{lemma}

\begin{proof}
We use that $u\circ\phi_l^{-1}:\cR^2\to\R$ minimizes $J$ over any compact set $\tilde K = \phi^{-1}_l(K)$ with $K \ss \cC_1$ compact in order to know that $u$ has enough regularity around $x_0$. The idea is that we apply the classical ACF monotonicity formula to $u^\pm$, centered at $x_0$ and, as the radius goes to zero, we measure the product of $|Du^\pm(x_0)|^2$.
\begin{align*}
2^{}|Du^+(x_0)|^2|Du^-(x_0)|^2 \leq 4\frac{1}{|x_0|^4}\int_{B_{|x_0|(x_0)}}|Du^+|^2\int_{B_{|x_0|(x_0)}}|Du^-|^2.
\end{align*}
Now we apply the ACF monotonicity formula given by Lemma \ref{lemma:ACF} with $\a = 4\pi/l\leq 2$. Notice that in order to apply such Lemma we are actually using Corollary \ref{coro:subharmonic}.
\begin{align*}
|Du^+(x_0)|^2|Du^-(x_0)|^2 &\leq 64|x_0|^{2\a-4}\frac{1}{|x_0|^{2\a}}\int_{\cC_{2|x_0|}}|Du^+|^2\int_{\cC_{2|x_0|}}|Du^-|^2,\\
&\leq C\|Du^+\|_{L^2(\cC_1)}^2\|Du^-\|_{L^2(\cC_1)}^2,\\
&\leq C.
\end{align*}
The minimizer $u$ also satisfies the Euler Lagrange equation at $x_0$ in the classical sense, $|Du^+(x_0)|^2 - |Du^-(x_0)|^2 = \l_+^2-\l_-^2 \neq 0$. Assume without lost of generality that $\l_+^2 - \l_-^2 = \L >0$. It implies
\begin{align*}
|Du^-(x_0)|^2 &\leq C\L^{-1},\\
|Du^+(x_0)|^2 &= |Du^-(x_0)|^2 + \L \leq C\L^{-1} + \L.
\end{align*}
\end{proof}

\begin{corollary}
Let $l \leq 2\pi$, $u$ be a minimizer of $J = J(\cC_1,\l_+,\l_-)$ with $\|Du\|_{L^2(\cC_1)} \leq 1$ and let $x_0 \in (\p\W^+\cup\p\W^-)\cap (\cC_{1/2}\sm\{0\})$; then $|Du(x_0)| \leq C$ for some universal constant $C>0$.
\end{corollary}

\begin{proof}
From the previous Lemma the only case left is when $x_0 \in (\p\W^+\D\p\W^-)\cap (\cC_{1/2}\sm\{0\})$. In such case $u$ keeps just one sign in a neighbirhood of $x_0$ (either non negative or non positive) and minimizes a one phase problem in the same neighborhood. From the gradient bound for the flat case we obtain the gradient bound, independent of the distance to the vertex.
\end{proof}

We split the proof of Theorem \ref{thm:optimal_reg_accute} into two Lemmas depending if $0\in\p\cC_1^\pm$ or not.

\begin{lemma}
Let $u$ and $x_0$ be as in Theorem \ref{thm:optimal_reg_accute} and assume additionally that $0\in\cC_1^+$. Then the same conclusions as in Theorem \ref{thm:optimal_reg_accute} hold.
\end{lemma}

\begin{proof}
Let $d = \dist(0,\p\cC_1^+) \in (0,1/2]$. The ball $\cC_d$ touches $\p\cC_{1/2}^+$ at some point $x_1$ where we know that $|Du^+(x_1)| \leq C$ from Lemma \ref{lemma:bound_gradient}. By using Harnack's inequality we get that $u(x) \geq Cu(0)$ in $B_{d/2}$ and then the following barrier can be put below $u$ in $\cC_d\sm\cC_{d/2}$ for $c$ sufficiently small,
\begin{align*}
\varphi(x) = cu(0)(\ln|x_0| - \ln|x|).
\end{align*}
This implies $C \geq |D\varphi(x_1)| = u(0)/d$. Which is the desired estimate at the origin.

For $x_0 \in \cC_{d/2}$ we use Harnack's inequality to get that $u(x_0) \leq Cd \leq C\dist(x_0,\p(\W^+\cap\cC_{1/2}))$. For $x_0$, now in $\cC_{d/8}\sm\{0\}$, we use the monotonicity of the Dirichlet energy,
\begin{align*}
|Du(x_0)|^2 &\leq \frac{1}{|x_0|^2}\int_{B_{|x_0|}(x_0)}|Du|^2,\\
&\leq 4\frac{1}{(2|x_0|)^2}\int_{\cC_{2|x_0|}}|Du|^2,\\
&\leq C\frac{1}{d^2}\|Du\|_{L^2(\cC_{d/4})}^2,\\
&\leq C\frac{1}{d^4}\|u\|_{L^2(\cC_{d/2})}^2,\\
&\leq C
\end{align*}
Which are the desired estimates at $\cC_{d/8}$.

Finally we consider $x_0 \in \cC_{1/4} \sm \cC_{d/8}$. Let $B_R(x_0)$ be the largest ball contained in $\cC_{1/2}^+\sm\{0\}$. If $B_R(x_0) \cap \p\cC_{1/2}^+ \ni x_1$ then the estimates for $x_0$ follow by using Lemma \ref{lemma:bound_gradient} at $x_1$ and considering a lower barrier as before. If $R = |x_0|$ 
we also use a similar barrier and instead of Lemma \ref{lemma:bound_gradient} we use the estimates just proved at $\cC_{d/2}$. Let,
\begin{align*}
\varphi(x) = cu(x_0)(\ln R - \ln|x|),
\end{align*}
with $c$ small enough such that by using Harnack's inequality we can get that $u \geq \varphi$ in $B_R(x_0)\sm B_{R/2}(x_0)$. Because $u^+ \leq Cd$ in $\cC_{d/2}$ we get that $C \geq |D\varphi(0)| \geq u(x_0)/R$. It implies that $u(x_0) \leq CR \leq C\dist(x_0,\p(\W^+\cap\cC_{1/2}))$. For the gradient estimate we can use interior estimates at $B_R(x_0)$, i.e. $|Du(x_0)| \leq C|u(x_0)|/R \leq C$.
\end{proof}

\begin{lemma}
Let $u$ and $x_0$ be as in Theorem \ref{thm:optimal_reg_accute} and assume additionally that $0\in\p\cC_1^+$. Then the same conclusions as in Theorem \ref{thm:optimal_reg_accute} hold.
\end{lemma}

\begin{proof}
The idea is to use a covering argument to pass the estimates from points that are close to $\p\cC_1^+$ to every other point in the positivity set.

Let $x_0 \in \cC_{1/4}^+\sm\{0\}$ and assume that for $r = |x_0|/2$, $B_r(x_0)\cap\{u^+=0\} \neq \emptyset$. Then the estimate follows as before by using Lemma \ref{lemma:bound_gradient} because for $d = \dist(x_0,\p\cC_{1/2}^+)$ the ball $B_d(x_0)$ doesn't contain the vertex.

For general $x_0 \in \cC_{1/4}^+\sm\{0\}$ we consider a finite covering of $\p\cC_r$ with balls center at $\cC_r$ and radius $r/2$ where $r=|x_0|$. Because $\{u = 0\}\cap\cC_{1/2} \neq \emptyset$ there is one of these balls that intersects $\{u^+ = 0\}$ and then the estimates are valid there. To obtain the estimates at $x_0$ we just need to apply Harnack's inequality in a finite chain of balls up to one that reaches $x_0$.

To conclude let us notice that the gradient is not necessarily well defined at the vertex because for $l \neq 2\pi$ the tangent space at $0$ is not well defined. Still the previous gradient estimate holds uniformly up to the vertex. 
\end{proof}

These two previous Lemmas conclude the proof of Theorem \ref{thm:optimal_reg_accute}. 


\subsubsection{Blows-up}

The first part in Corollary \ref{coro:blowup} follows from the previous stability and optimal regularity.

\begin{proof}[Proof of the first part in Corollary \ref{coro:blowup}]
Let $K\ss\cC$ be a compact set. By the scaling of the functional we have that $\|u_k\|_{H^1(2K)}$ is uniformly bounded starting at some $k_0$ sufficiently large. There exists then an accumulation point $u\in H^1(2K)$ which is also a minimizer in $K$ by Theorem \ref{thm:stability}. Moreover the whole sequence converges uniformly to $u$ in $K$ by the same Theorem. By the definition of the rescaling we have that the same sequence is uniformly bounded in $C^{0,1}(K)$ and therefore there is an acculumation point in $C^{0,1}(K)$. Because the sequence already converged to $u$ uniformly we conclude that $u \in C^{0,1}(K)$ and the convergence happened also in $C^{0,\beta}(K)$ for $\beta < 1$.
\end{proof}

For the second part in Corollary \ref{coro:blowup} we need to use a monotonicity formula as in \cite{MR1759450}. There is also a similar monotonicity formula in \cite{MR1620644}. The proofs of such monotonicity formulas use radial variations which naturally adapt to our situation with a cone.  We reproduce the proof in \cite{MR1759450} here. Notice also that the proof works no matter the length $l$ of the cone, however for $l\geq2\pi$ with $0 \in \partial\cC^+ \cap \partial \cC^-$, since the optimal reguarity is unknown it might be possible for $W(\cC_r, u)= -\infty$. 

\begin{theorem} \label{tim:weiss} Let $u$ be a minimizer of $J(\cC_1,\l_+,\l_-)$ such that $u(0)=0$ and define the \emph{Weiss energy} for $r\in(0,1)$,
\begin{align*}
W(\cC_r,u) &= \frac{1}{r^2}J(\cC_r,u) - \frac{1}{r}\int_0^r\int_{\p\cC_1}\1Du(tx)\cdot x\2^2 dt,\\
&= \frac{1}{r^{2}} \left( \int_{\C_r}{|Du|^2 + \l_+ \chi_{\{u>0\}} + \l_- \chi_{\{u<0\} } } \right),\\
&{} -\frac{1}{r}\int_0^r\int_{\p\cC_1}\1Du(tx)\cdot x\2^2 dt
\end{align*}  
Then $W(\cC_r,u)$ is monotone increasing in $r$. Furthermore, if $0<r_1 < r_2<1$, then $W(\cC_{r_1}, u) = W(\cC_{r_2},u)$ if and only if $u$ is homogeneous of degree $1$ with respect to $0$ on the ring $\cC_{r_2}\sm\cC_{r_1}$.
\end{theorem}

\begin{remark}\label{rmk:rescaling_of_W}
For $u_r = r^{-1}u(r\cdot)$, the functional $W$ enjoys the following rescaling property:
\begin{align*}
W(\cC_R,u_r)= W(\cC_{rR},u_r)
\end{align*}
\end{remark}

\begin{proof}
An admissible competitor against $u$ in $\cC_t$ is given by the following 1-homogeneous function constructed from the trace of $u$ in $\p\cC_t$. For $x\neq 0$ we denote $\vec{\theta} = x/|x|$,
\begin{align*}
u_t(x) = \frac{|x|}{t}u\1t\vec{\theta}\2.
\end{align*}
Moreover by Proposition \ref{prop:Lipschitz_degenerate} we have that $u$ is Lipschitz in $\bar\cC_t$ so that the following computation is well justified.
\begin{align*}
&\int_{\cC_r} |Du_t|^2 + \l_+\chi_{\{u_t>0\}} + \l_-\chi_{\{u_t,0\}},\\
&= \frac{t}{2}\int_{\p\cC_t}|Du|^2-\1Du\cdot \vec{\theta}\2^2 + \frac{u^2}{t^2} + \l_+\chi_{\{u>0\}} +  \l_-\chi_{\{u<0\}}
\end{align*}

Notice, on the other hand, that the derivative of $t^{-2}J(\cC_t)$ with respect to $t$ throws out some similar terms to the ones we already have above,
\begin{align*}
\frac{t^3}{2}(t^{-2}J(\cC_t))' = -J(\cC_t) + \frac{t}{2}\int_{\p\cC_t}|Du|^2 + \l_+\chi_{\{u>0\}} +  \l_-\chi_{\{u<0\}}.
\end{align*}
Then by using that $u$ is a minimizer in $\cC_t$ we obtain,
\begin{align*}
0 \leq J(u_t) - J(u) &= -J(\cC_t) + \frac{t}{2}\int_{\p\cC_t}|Du|^2+ \l_+\chi_{\{u>0\}} + \l_-\chi_{\{u<0\}},\\
&{} - \frac{t}{2}\int_{\p\cC_t}\1Du\cdot \vec{\theta}\2^2 - \frac{u^2}{t^2},\\
&= \frac{t^3}{2}(t^{-2}J(\cC_t))' - \frac{t}{2}\int_{\p\cC_t}\1Du\cdot \vec{\theta}\2^2 - \frac{u^2}{t^2}
\end{align*}

By writing,
\begin{align*}
\frac{u(t\vec{\theta})}{t} = \lim_{\e\to0^+}\frac{1}{t}\int_\e^t Du\1s\vec{\theta}\2\cdot \vec{\theta} ds \text{ for $t\vec{\theta} \in\p\cC_t$},
\end{align*}
we obtain by H\"older's inequality that
\begin{align*}
0 &\leq \lim_{\e\to0^+}\frac{1}{t}\int_{\p\cC_1}\frac{1}{t}\int_\e^t \1Du(s\vec{\theta})\cdot\vec{\theta}\2^2 ds - \1\frac{1}{t}\int_\e^t Du\1s\vec{\theta}\2\cdot \vec{\theta}\2^2,\\
&\leq (t^{-2}J(\cC_t))' + \lim_{\e\to0^+}\int_{\p\cC_1}\3\frac{1}{t^2}\int_\e^t \1Du(s\vec{\theta})\cdot\vec{\theta}\2^2 ds - \frac{1}{t}\1Du(t\vec{\theta})\cdot\vec{\theta}\2^2\4,\\
&= \1t^{-2}J(\cC_t) - \lim_{\e\to0^+}\frac{1}{t}\int_\e^t \int_{\p\cC_1}\1Du(s\vec{\theta})\cdot\vec{\theta}\2^2 ds\2'
\end{align*}
This implies the monotonicity.

In case of having $W(\cC_{r_1}) = W(\cC_{r_2})$ the motononicity forces the equality also in the whole interval $[r_1,r_2]$. The use of H\"older's implies that for almost every $s\in[r_1,r_2]$, $Du(s\vec{\theta})\cdot\vec{\theta}$ is independent of $s$ which is equivalent to the radial derivative of $u$ being 0-homogeneous and $u$ being 1-homogeneous. 
\end{proof}

\begin{proof}[Proof of the second part in Corollary \ref{coro:blowup}]
All we have to check is that for any $r_1 < r_2$, $W(\cC_{r_1},u_0) \geq W(\cC_{r_2},u_0)$. From the rescaling property of the functional, Remark \ref{rmk:rescaling_of_W}, we obtain that for $i=1,2$ we have that $W(\cC_{\r_kr_i},u) = W(\cC_{r_i},u_{\r_k}) \to W(\cC_{r_i},u_0)$. For each $k$ there exists some $m_k \geq k$ such that $\r_{m_k}r_2 < \r_kr_1$ which implies that $W(\cC_{\r_{m_k}r_2},u) \leq W(\cC_{\r_kr_1},u)$. By taking $k\to\8$ we conclude the desired inequality.
\end{proof}

\bibliographystyle{plain}
\bibliography{./mybibliography}

\end{document}